\documentclass[a4paper,11pt]{article}
\usepackage{graphicx,amssymb} 
\usepackage{amsmath}
\usepackage{amsthm}

\theoremstyle{definition} \newtheorem{theorem}{Theorem}
\theoremstyle{definition} 
\theoremstyle{definition} \newtheorem{lemma}{Lemma}
\theoremstyle{definition} 
\theoremstyle{definition} \newtheorem{proposition}{Proposition}

\title{The Maximal Graph Dirichlet Problem in Semi-Euclidean Spaces}
\date{}
\author{Benjamin Stuart Thorpe}

\begin{document}

\maketitle

\begin{abstract}
The maximal graph Dirichlet problem asks whether there exists a spacelike graph, in a semi-Euclidean space, with a given boundary and with mean curvature everywhere zero. We prove the existence of solutions to this problem under certain assumptions on the given boundary. Most importantly, the results proved here will hold for graphs of codimension greater than $1$. 
\end{abstract}

\section{Introduction}
\label{sec:Introduction}
Given a system of partial differential equations, a domain and a function defined on its boundary, a \emph{Dirichlet problem} is the question of whether there exists a solution to the system, on this domain, with the given boundary values. It is well known that a submanifold of a Euclidean space with mean curvature zero everywhere is called \emph{minimal}. For a submanifold that can be written as a graph over some domain, the mean curvature zero condition is equivalent to a second order, elliptic, quasilinear system of partial differential equations. We can therefore consider the Dirichlet problem for the minimal graph system, which is related to the variational problem of minimizing the volume functional for graphs with a given boundary.

For semi-Euclidean spaces (e.g. Minkowski space, which is well known from Relativity), we are again interested in submanifolds with mean curvature zero, but now we also assume that these submanifolds are \emph{spacelike} (have positive definite induced metric). We call such submanifolds \emph{maximal}. For graphs, the spacelike condition is equivalent to a bound on the gradient, and the mean curvature zero condition is again an elliptic, quasilinear system of equations (ellipticity comes from the spacelike condition). In this case we can consider a maximal graph Dirichlet problem, which is related to the problem of \emph{maximizing} the volume functional for spacelike graphs with a given boundary.    

In Minkowski space, the codimension $1$ maximal graph Dirichlet problem was first dealt with in \cite{FJF}, by using boundary conditions similar to those used in the minimal graph case (see \cite{JS}). A more general existence theorem was then proved in \cite{BS}. Both of these results are for codimension 1 graphs, where the mean curvature zero condition is a single equation. The proofs involve using the assumptions on the boundary to obtain gradient estimates, which then give existence of solutions by the usual methods for elliptic problems (using Schauder fixed point theorem, as in chapter 11 of \cite{GT}).

For graphs of higher codimension, these problems are more difficult since we have to consider systems of equations. In \cite{LO}, Lawson and Osserman give an example which shows that the minimal graph problem is sometimes not solvable for codimension greater than $1$, even for very `nice' domains and boundary data. But there are still some examples of existence theorems for the higher codimension minimal graph Dirichlet problem. For example, we can use the inverse function theorem to prove the existence of solutions whenever the $C^{2,\alpha}$ norm of the boundary data is less than some (unknown) constant.\footnote{See Theorem 4.2 of L~Martinazzi, {\sl The Non-Parametric Problem of Plateau in Arbitrary Codimension}, arXiv:math/0411589v1[math.AP] (Nov 2004).} 

The main result claimed in \cite{MTW2} is more interesting, and roughly says that a solution exists if the domain is convex and the $C^{2}$ norm of the boundary data is less than some known constant (which depends only on the diameter and dimension of the domain). The proof involves mean curvature flow, but the boundary $C^{1,\alpha}$ estimate was overlooked there, as we will explain later (section \ref{TheMinimalGraphProblem}).

For the maximal graph problem with codimension $n \geq 2$ in $\mathbb{R}^{m+n}_{n}$, very little is known. This is the case that we will consider. Unlike the higher codimension existence theorems mentioned above (in the Euclidean case), we will use standard elliptic methods by proving a suitable gradient estimate. However, we will have to deal with the fact that the higher order estimates that hold for single equations do not necessarily hold for systems. For this reason, we will only prove existence theorems either in the case of graphs with dimension $m = 2$, or for $m \geq 2$ when the gradient estimate is sufficiently strong. We will prove that solutions to the maximal graph Dirichlet problem exist whenever the domain is convex and the $C^{2}$ norm of the boundary data is small enough (see Theorem \ref{mainmaxthm} when $m=2$ and Theorem \ref{arbdimcodimexistthm} when $m\geq 2$).

\section{Preliminaries}
\label{sec:Preliminaries}
For integers $m\geq 2$ and $n \geq 1$, we will denote by $\mathbb{R}^{m+n}_{n} = (\mathbb{R}^{m+n}, \left\langle\cdot , \cdot\right\rangle)$ the \emph{semi-Euclidean} space with metric tensor $\left\langle v , w \right\rangle = \sum^{m}_{i=1}v^{i}w^{i} - \sum^{m+n}_{\gamma=m+1}v^{\gamma}w^{\gamma}$. Given a submanifold, we can take the induced metric from $\mathbb{R}^{m+n}_{n}$ in the usual way, and we say that the submanifold is \emph{spacelike} if the induced metric is positive definite. As always, we define the second fundamental form by $B(v, w) = \bar{\nabla}_{v}w - \nabla_{v}w$ for tangent vector fields $v$ and $w$, where $\bar{\nabla}$ and $\nabla$ are the Levi-Civita connections on $\mathbb{R}^{m+n}_{n}$ and the submanifold respectively. Taking the trace of $B$ with respect to the induced metric $g$ gives the \emph{mean curvature} $H$ on the submanifold (which will be a normal vector field). We can even define the  gradient, divergence and induced Laplace operator on such submanifolds in the usual way (see \cite{BO} for details).

We will consider submanifolds that can be written as graphs over a domain $\Omega$ in $\mathbb{R}^{m}$, $M = \{(x,u(x))\in \mathbb{R}^{m+n}_{n} \; | \; x\in\Omega\}$ for some smooth $u: \Omega\rightarrow \mathbb{R}^{n}$. The induced metric on the graph will be given by the matrix $g = I - Du^{T}Du$. It will be convenient for us to use the following norms for the maps $Du(x):\mathbb{R}^{m}\rightarrow \mathbb{R}^{n}$ and $D^{2}u(x):\mathbb{R}^{m}\times\mathbb{R}^{m}\rightarrow\mathbb{R}^{n}$,
\begin{eqnarray}
|||Du|||(x) = \sup_{|v|=1}|Du(x)(v)|  & \textrm{ and } & |||D^{2}u|||(x) = \sup_{|v|=1}|D^{2}u(x)(v,v)|,\nonumber
\end{eqnarray}
where $|\cdot|$ denotes the usual Euclidean norm. It is possible to show that $|||Du|||^{2}$ will be equal to the largest eigenvalue of $Du^{T}Du$ at each point, and that $|||Du||| \leq |Du| \leq \sqrt{m}|||Du|||$. Using the obvious relationship between $|||Du|||$ and the eigenvalues of $g$, we see that the graph will be spacelike if and only if $|||Du||| < 1$. Also, for any $0<C<1$, we have $\sqrt{\det g} \geq C  \Rightarrow  |||Du|||^{2} \leq 1 - C^{2}$ and $|||Du|||^{2} \leq C  \Rightarrow  \sqrt{\det g} \geq (1 - C)^{m/2}$.

As in the well-known Euclidean case, it is easy to check that the mean curvature vector of a graph is given by applying the induced Laplace operator on the graph (which we denote by $\Delta_{M}$) to the vector $(x,u(x))$. This tells us that the mean curvature vector is zero if and only if $g^{ij}(Du)\partial^{2}u/\partial x^{i}\partial x^{j} = 0$.\footnote{We denote by $g_{ij} = \delta_{ij} - (\partial u^{\gamma}/\partial x^{i})(\partial u^{\gamma}/\partial x^{j})$ the components of $g$, and by $g^{ij}$ the components of its inverse. We always use the summation convention over repeated indices $i, j,\ldots \in\{1, \ldots, m\}$  and $\gamma,\nu,\ldots \in \{m+1, \ldots, m+n\}$.} This is a quasilinear elliptic system of $n$ equations for $u$. Given a bounded domain $\Omega$ in $\mathbb{R}^{m}$ and boundary data $\phi:\partial\Omega\rightarrow\mathbb{R}^{n}$, we would therefore like to prove the existence of a smooth solution to the following Dirichlet problem:
\begin{eqnarray}
g^{ij}(Du)\frac{\partial^{2}u}{\partial x^{i}\partial x^{j}} = 0 \textrm{ and } |||Du||| < 1 \textrm{ in } \Omega, & u = \phi \textrm{ on } \partial\Omega,\nonumber
\end{eqnarray}  
where $u$ is at least $C^{2}$ in $\Omega$ and $C^{0}$ on $\bar{\Omega}$. In particular, we will look for solutions in H\"older spaces. For $\alpha \in (0,1)$ we say that a function $u:\Omega \rightarrow \mathbb{R}^{n}$ lies in the H\"older space $C^{k,\alpha}(\bar{\Omega};\mathbb{R}^{n})$ if and only if it is in $C^{k}(\bar{\Omega};\mathbb{R}^{n})$ and $||u||_{k,\alpha}  =  ||u||_{k} + \sup_{x\neq y}  |D^{k}u(x) - D^{k}u(y)|/|x - y|^{\alpha} =  ||u||_{k} + [D^{k}u]_{\alpha}$ is finite, where $||u||_{k}$ is the usual $C^{k}$ norm. Note that, with the norm $||\cdot||_{k,\alpha}$, the space $C^{k,\alpha}(\bar{\Omega};\mathbb{R}^{n})$ will be a Banach space. We will sometimes just call functions in these spaces $C^{k,\alpha}$ functions. Functions that are $C^{k,\alpha}$ on compact subsets of a domain will be called \emph{locally} $C^{k,\alpha}$ on the domain. 

We will need the following fact, which uses Leray-Schauder fixed point theorem to get existence of solutions to our problem under the assumption of some a priori estimates. The proof is similar to Theorem 11.4 of \cite{GT}, but slightly more difficult here since $g^{ij}(Dw)$ is only positive definite when the graph of $w$ is spacelike. This is why we make a more complicated assumption on the gradient and why we use the set $R$, to avoid non-spacelike graphs. 

\begin{lemma}\label{firstlemma}
For some $\alpha \in (0,1)$, let $\Omega$ be a bounded $C^{2,\alpha}$ domain in $\mathbb{R}^{m}$ and let $\phi \in C^{2,\alpha}(\bar{\Omega}; \mathbb{R}^{n})$. Suppose that there exist constants $\kappa \in (0,1)$ and $C>0$ such that $\sup_{\Omega}|||D\phi|||^{2} \leq 1- \kappa$, and such that
\begin{eqnarray}
\sup_{\Omega}|||Du|||^{2} < 1 - \kappa & \textrm{and} & ||u||_{1,\alpha}\leq C\nonumber
\end{eqnarray}
whenever $u \in C^{2,\alpha}(\bar{\Omega};\mathbb{R}^{n})$ gives a maximal graph in $\mathbb{R}^{m+n}_{n}$ with  $\sup_{\Omega}|||Du|||^{2} \leq 1 - \kappa$ and $u|_{\partial\Omega} = \sigma\phi|_{\partial\Omega}$ for some $\sigma \in [0,1]$. Then there exists a solution $u\in C^{2,\alpha}(\bar{\Omega};\mathbb{R}^{n})$ to the maximal graph Dirichlet problem in $\mathbb{R}^{m+n}_{n}$ with $u|_{\partial\Omega} = \phi|_{\partial\Omega}$.
\end{lemma}

\begin{proof}
We let $R = \{u \in C^{1,\alpha}(\bar{\Omega};\mathbb{R}^{n}) \;| \; |||Du|||^{2} \leq 1 - \kappa\}$ and we define maps $f:C^{1,\alpha}(\bar{\Omega};\mathbb{R}^{n})\rightarrow R$ and $T:R\rightarrow C^{1,\alpha}(\bar{\Omega};\mathbb{R}^{n})$. For any $v\in C^{1,\alpha}(\bar{\Omega};\mathbb{R}^{n})$, we take $f(v)$ to be equal to $v$ when $\sup_{\Omega}|||Dv|||^{2} \leq 1 - \kappa$ and equal to $(1-\kappa)^{1/2}v/\sup_{\Omega}|||Dv|||$ otherwise. For any $w \in R$, we define $T(w)$ to be the unique solution $u$ to the system of $n$ linear Dirichlet problems given by $g^{ij}(Dw)\partial^{2}u/\partial x^{i}\partial x^{j} = 0$ in $\Omega$ with $u = \phi$ on $\partial\Omega$. Since $w \in R$ implies that the system is elliptic and that the coefficients $g^{ij}(Du)$ are $C^{0,\alpha}$ functions, we know that such a solution must exist in $C^{2,\alpha}(\bar{\Omega};\mathbb{R}^{n})$ by the usual existence theorem for linear equations (see Theorem 6.14 of \cite{GT}).

We claim that $\tilde{T} = T\circ f : C^{1,\alpha}(\bar{\Omega};\mathbb{R}^{n}) \rightarrow C^{1,\alpha}(\bar{\Omega};\mathbb{R}^{n})$ will be continuous and compact (i.e. the images of bounded sets are precompact).
The map $f$ is continuous and clearly maps bounded sets to bounded sets (with respect to the $C^{1,\alpha}$ norm). By the Schauder estimates (see Theorem 6.6 of \cite{GT}), sets in $R$ with bounded $C^{1,\alpha}$ norm are mapped by $T$ to sets with bounded $C^{2,\alpha}$ norm. But, by the Arzela-Ascoli theorem, bounded sets in $C^{2,\alpha}(\bar{\Omega};\mathbb{R}^{n})$ are precompact in $C^{2}(\bar{\Omega};\mathbb{R}^{n})$ and $C^{1,\alpha}(\bar{\Omega};\mathbb{R}^{n})$. Continuity of $T$ can be proved exactly as in the proof of Theorem 11.4 of \cite{GT}.
          
Now we need to make use of the estimates that we have assumed to exist. Suppose that, for $\sigma \in [0,1]$, we have $v \in C^{1,\alpha}(\bar{\Omega};\mathbb{R}^{n})$ with $\sigma \tilde{T}(v)=v$. We have two possible cases. First, if $\sup_{\Omega}|||Dv|||^{2} > 1 - \kappa$ then we have $\sigma T (v\sqrt{1-\kappa} / \sup_{\Omega}|||Dv|||) = v$, so $w = v\sqrt{1-\kappa}  / \sup_{\Omega}|||Dv|||$ solves the maximal graph Dirichlet problem with $w = (\sigma\sqrt{1-\kappa}/\sup_{\Omega}|||Dv|||)\phi$ on the boundary. But $(\sigma\sqrt{1-\kappa}/\sup_{\Omega}|||Dv|||) \in [0,1]$, so the assumptions that we make here imply that $\sup_{\Omega}|||Dw|||^{2} < 1 - \kappa$, which contradicts the fact that $\sup_{\Omega}|||Dw|||^{2} = 1 - \kappa$. Therefore we only need to consider the case $\sup_{\Omega}|||Dv|||^{2} \leq 1 - \kappa$, where $v$ will be a fixed point of $\sigma T$ and will be a solution of the maximal graph Dirichlet problem with boundary values $\sigma \phi$. Our assumptions now imply that $||v||_{1,\alpha} \leq C$.

We conclude that $\tilde{T}$ is a compact map from the Banach space $C^{1,\alpha}(\bar{\Omega};\mathbb{R}^{n})$ into itself and, for any $\sigma \in [0,1]$, any fixed point $v$ of $\sigma\tilde{T}$ satisfies $||v||_{1,\alpha} \leq C$. A version of the Schauder fixed point theorem (see Theorem 11.3 of \cite{GT}) tells us that $\tilde{T}$ has a fixed point. As explained above, but now just taking $\sigma=1$, this fixed point must have $|||Du|||^{2}<1-\kappa$. This will be a spacelike solution of the maximal graph Dirichlet problem with boundary values given by $\phi$.
\end{proof}

It is important to note that any $C^{2,\alpha}$ solution to a maximal graph Dirichlet problem (as given by the lemma) will be smooth on $\bar{\Omega}$ if the domain and boundary data are both smooth. This is proved by induction using Theorem 6.19 of \cite{GT}, which says that if $u$ is a $C^{k,\alpha}$ solution then the coefficients $g^{ij}(Du)$ are $C^{k-1,\alpha}$ and therefore $u$ must be $C^{k+1,\alpha}$.\footnote{Instead of 6.19, we could have used Theorem 6.17 of \cite{GT}, which says that a \emph{locally} $C^{k,\alpha}$ solution will be locally $C^{k+1,\alpha}$. This can be used even when the domain and data are not smooth, but only gives smoothness of solutions on the interior.} We also note that any solution with boundary data $\phi$ will have $|u|$ uniformly bounded in terms of $\sup_{\Omega}|\phi|$. This follows directly from the elliptic maximum principle.
 
\section{Gradient Estimate}
\label{sec:GradientEstimate}
 
Given some $\kappa \in (0,1)$, we will find conditions on $\Omega \subset \mathbb{R}^{m}$ and $\phi : \bar{\Omega}\rightarrow \mathbb{R}^{n}$ such that any smooth solution to the corresponding maximal graph Dirichlet problem with $\sup_{\Omega}|||Du|||^{2} \leq 1 - \kappa$ must satisfy $\sup_{\Omega}|||Du|||^{2} < 1 - \kappa$. We will assume that $\Omega$ and $\phi$ are both $C^{2}$, and that $\Omega$ is bounded and convex. We follow the method from \cite{MTW2}, but note that we have an elliptic (rather than parabolic) system here. It is also important to note that the inequalities that hold for maximal graphs are slightly different to the corresponding inequalities for minimal graphs, so $\kappa$ will appear in our estimates in a different way, affecting the assumptions that we need to make.  

Given such a solution $u$, the first step is to define the linear elliptic operator $L = g^{ij}(Du)\partial^{2}/\partial x^{i}\partial x^{j}$. Then, fixing any $\gamma \in \{m+1,\ldots,m+n\}$ and any $p\in\partial\Omega$, we define functions $S^{\pm}:\bar{\Omega} \rightarrow \mathbb{R}$ by $S = \nu\log(1 + \zeta d) \mp (u^{\gamma} - \phi^{\gamma})$. Here $d(x)$ is the distance from any point $x\in\bar{\Omega}$ to the $(m-1)$-dimensional hyperplane tangent to $\partial\Omega$ in $\mathbb{R}^{m}$ at $p$. The positive constants $\nu$ and $\zeta$ will be chosen later. Using the obvious facts that $Lu = 0 = Ld$, $d(x) \leq |x - p| \leq \textrm{diam}\Omega$, $|Dd| = 1$ and that the eigenvalues of $g^{-1}$ are between $1$ and $1/\kappa$ (by the assumed gradient bound), we have
\begin{equation}
LS^{\pm} = \frac{-\nu\zeta^{2}}{(1 + \zeta d)^{2}}g^{ij}\frac{\partial d}{\partial x^{i}}\frac{\partial d}{\partial x^{j}} \pm L\phi^{\gamma}\leq \frac{-\nu\zeta^{2}}{(1 + \zeta\textrm{diam}\Omega)^{2}} + \frac{m}{\kappa}|||D^{2}\phi|||.\label{constraint}
\end{equation}
If we assume that the right hand side is non-positive, then $LS^{\pm} \leq 0$ and we can apply the elliptic maximum principle to see that the infimum of $S^{\pm}$ occurs on the boundary. But it is clear that $S^{\pm}\geq 0$ on the boundary (by the Dirichlet condition), so we have $S^{\pm} \geq 0$ on all of $\Omega$ and therefore $\nu\log(1 + \zeta d) \geq |u^{\gamma} - \phi^{\gamma}|$. From here, we can directly follow the steps in \cite{MTW2} to get a gradient estimate of the form $|||Du||| \leq \nu\zeta + 2|||D\phi|||$ at the point $p$, and hence at any boundary point. To minimize $\nu\zeta$ in such a way that the right hand side of (\ref{constraint}) is non-positive, we take $\zeta = 1/\textrm{diam}\Omega$ and $\nu\zeta = 4m\textrm{diam}\Omega\sup_{\Omega}|||D^{2}\phi|||/\kappa$. With this choice of constants, we have the boundary estimate
\begin{equation}
\sup_{\partial\Omega}|||Du|||\leq \frac{4m\textrm{diam}\Omega}{\kappa}\sup_{\Omega}|||D^{2}\phi||| + 2\sup_{\partial\Omega}|||D\phi|||.\nonumber
\end{equation}
If the right hand side is small enough, then this will give a gradient estimate on the full domain. We see this by using inequality 4.6 of \cite{LS}, which implies that $\Delta_{M}\log\sqrt{\det g} \leq 0$. Applying the maximum principle to this tells us that $\sqrt{\det g} \geq \inf_{\partial\Omega}\sqrt{\det g}$ on $\Omega$. Therefore, if $|||Du|||^{2}< 1-\kappa^{1/m}$ on $\partial\Omega$ then $\sqrt{\kappa}< \inf_{\partial\Omega}\sqrt{\det g} \leq \sqrt{\det g}$, which implies that $|||Du|||^{2} < 1 - \kappa$ on $\Omega$. 

\begin{proposition}
Let $\Omega$ be a bounded, convex, $C^{2}$ domain in $\mathbb{R}^{m}$ and let $\phi:\bar{\Omega}\rightarrow \mathbb{R}^{n}$ be a $C^{2}$ function. Assume, for some $\kappa\in(0,1)$, that $\phi$ satisfies
\begin{equation}
\frac{4m\textrm{\emph{diam}}\Omega}{\kappa}\sup_{\Omega}|||D^{2}\phi||| + 2\sup_{\partial\Omega}|||D\phi||| < \sqrt{1 - \kappa^{1/m}}. \label{boundaryassumpt}
\end{equation}
If $u$ is a smooth solution of the corresponding maximal graph Dirichlet problem in $\mathbb{R}^{m+n}_{n}$, and if $\sup_{\Omega}|||Du|||^{2} \leq 1 - \kappa$, then $\sup_{\Omega}|||Du|||^{2} < 1 - \kappa$.   
\end{proposition}

\section{An Existence Theorem in $\mathbb{R}^{2+n}_{n}$}
\label{sec:AnExistenceTheoremInMathbbR2NN}
The previous section gives us a gradient estimate as required in Lemma \ref{firstlemma}. Since we are only interested in codimension $n \geq 2$, we have a system of equations, and therefore the $C^{1,\alpha}$ estimates used for single equations are not available. But, for $2$-dimensional graphs in $\mathbb{R}^{2+n}_{n}$, we can use strong a priori estimates that hold for \emph{linear} elliptic equations in two variables. In particular, we use the next lemma, which follows directly from a comment on page 304 of \cite{GT}. 

\begin{lemma}\label{2dimest}
Let $\Omega$ be a smooth, bounded domain in $\mathbb{R}^{2}$ and let $\phi\in C^{\infty}(\bar{\Omega})$. Let $L=a^{ij}(x)\partial^{2}/\partial x^{i}\partial x^{j}$, where the matrix given by $a^{ij}(x)\in C^{\infty}(\Omega)$ has positive eigenvalues $\lambda(x) \leq \Lambda(x)$ such that $\Lambda/\lambda \leq\eta$ for some constant $\eta$. Suppose that $u \in C^{2}(\bar{\Omega})$ is a solution of $Lu = 0$ with $u = \phi$ on $\partial\Omega$. Then there exist constants $\alpha(\eta,\Omega)\in(0,1)$ and $C(\eta,\Omega,||\phi||_{2})>0$ such that $||u||_{1,\alpha} \leq C$.
\end{lemma}

\begin{theorem}\label{mainmaxthm}
Let $\Omega$ be a smooth, convex and bounded domain in $\mathbb{R}^{2}$. Let $\phi : \bar{\Omega} \rightarrow \mathbb{R}^{n}$ be a smooth function satisfying, for some $\kappa \in (0,1)$, inequality (\ref{boundaryassumpt}) with $m = 2$ and $\sup_{\Omega}|||D\phi|||^{2} \leq 1 - \kappa$. Then there exists a smooth solution $u$ to the maximal graph Dirichlet problem in $\mathbb{R}^{2+n}_{n}$ with $|||Du|||^{2} < 1 - \kappa$ on $\bar{\Omega}$ and $u = \phi$ on $\partial\Omega$.   
\end{theorem}
The result claimed in the introduction, when $m=2$, clearly follows from this since the assumptions on $\phi$ are satisfied whenever its $C^{2}$ norm is small enough.
\begin{proof}
If $u$ is a maximal graph with boundary values $\sigma\phi$ (for some $\sigma\in [0,1]$) and $\sup_{\Omega}|||Du|||^{2}\leq 1 - \kappa$, then the gradient estimate from the previous section gives $\sup_{\Omega}|||Du|||^{2} < 1 - \kappa$. The eigenvalues of $g^{ij}$ are between $1$ and $1/\kappa$, so we can take $\eta = 1/\kappa$ in Lemma \ref{2dimest} to get a $C^{1,\alpha}$ estimate, allowing us to apply Lemma \ref{firstlemma} to prove the theorem.
\end{proof}

\section{An Existence Theorem in $\mathbb{R}^{m+n}_{n}$}
\label{sec:AnExistenceTheoremInMathbbRMNN}

Now we prove $C^{1,\alpha}$ estimates, for $m\geq 2$ and $n\geq 1$, which will be used to prove an existence theorem for our problem in $\mathbb{R}^{m+n}_{n}$. We hope to use methods from chapter 13 of \cite{GT}. The $C^{1,\alpha}$ estimates there do not apply directly to systems, so we will need to assume that our gradient estimate is even stronger than before (we will explain why later). We do not expect to get the most general results possible, so we will prove the $C^{1,\alpha}$ estimates as quickly as possible, without wasting time trying to get the best estimates at each step. In this section, $\Omega$ will be a smooth, bounded domain in $\mathbb{R}^{m}$, and $\phi:\bar{\Omega}\rightarrow \mathbb{R}^{n}$ will be smooth with $||\phi||_{2}\leq \Phi_{2}$ and $||\phi||_{3}\leq\Phi_{3}$ for some constants $\Phi_{2},\Phi_{3}>0$. We assume that a smooth function $u:\bar{\Omega}\rightarrow \mathbb{R}^{n}$ is a solution to the corresponding maximal graph Dirichlet problem, with $|||Du|||^{2} \leq 1-\kappa$ on $\bar{\Omega}$ for some constant $\kappa \in (0,1)$. 

For any $\gamma \in \{m+1,\ldots,m+n\}$ and $r \in \{1,\ldots,m\}$, we define $w = \zeta (1-\kappa)^{1/2}\partial u^{\gamma}/\partial x^{r} + v$, where $\zeta$ is some constant (depending on $m$ and $n$) to be chosen later, and $v = \sum_{j, \nu}\left(\partial u^{\nu}/\partial x^{j}\right)^{2}$. Writing $g(p) = I - p^{T}p$ for $p = (p^{\nu}_{k})$,
\begin{eqnarray}
\frac{\partial}{\partial x^{i}}\left(g^{ij}(Du)\frac{\partial w}{\partial x^{j}}\right) 
& = & \frac{\partial g^{ij}}{\partial p_{k}^{\nu}}(Du)\frac{\partial^{2} u^{\nu}}{\partial x^{i}\partial x^{k}}\frac{\partial w}{\partial x^{j}} \nonumber\\
&& + g^{ij}(Du)\left(\zeta\sqrt{1-\kappa}\frac{\partial^{3}u^{\gamma}}{\partial x^{r}\partial x^{i}\partial x^{j}}\right) \nonumber\\
&& + g^{ij}(Du)\left(2\frac{\partial^{2} u^{\nu}}{\partial x^{i}\partial x^{k}}\frac{\partial^{2} u^{\nu}}{\partial x^{j}\partial x^{k}} + 2\frac{\partial^{3} u^{\nu}}{\partial x^{k}\partial x^{i} \partial x^{j}}\frac{\partial u^{\nu}}{\partial x^{k}}\right).\nonumber
\end{eqnarray}
But, since $g^{ij}(Du)\partial^{2}u/\partial x^{i}\partial x^{j} = 0$,  
\begin{eqnarray}
g^{ij}\frac{\partial^{3} u^{\nu}}{\partial x^{h}\partial x^{i}\partial x^{j}} & = & \frac{\partial}{\partial x^{h}}\left( g^{ij}\frac{\partial^{2} u^{\nu}}{\partial x^{i}\partial x^{j}} \right) - \frac{\partial }{\partial x^{h}}\left(g^{ij}\right)\frac{\partial^{2} u^{\nu}}{\partial x^{i}\partial x^{j}}\nonumber\\
& = & - \frac{\partial g^{ij}}{\partial p_{k}^{\delta}}\frac{\partial^{2} u^{\delta}}{\partial x^{k}\partial x^{h}}\frac{\partial^{2} u^{\nu}}{\partial x^{i}\partial x^{j}} \nonumber,
\end{eqnarray}
and therefore
\begin{eqnarray}
\frac{\partial }{\partial x^{i}}\left(g^{ij}\frac{\partial w}{\partial x^{j}} \right) &= & \frac{\partial g^{ij}}{\partial p_{k}^{\nu}}\frac{\partial^{2} u^{\nu}}{\partial x^{i}\partial x^{k}}\frac{\partial w}{\partial x^{j}} 
 - \zeta \sqrt{1-\kappa}\frac{\partial g^{ij}}{\partial p_{k}^{\delta}}\frac{\partial^{2}u^{\delta} }{\partial x^{k}\partial x^{r}}\frac{\partial^{2}u^{\gamma}}{\partial x^{i}\partial x^{j}}\nonumber\\
& & + 2g^{ij}\frac{\partial^{2} u^{\nu} }{\partial x^{i} \partial x^{h}}\frac{\partial^{2}u^{\nu} }{\partial x^{j}\partial x^{h}} - 2\frac{\partial u^{\nu}}{\partial x^{h}}\frac{\partial g^{ij}}{\partial p_{k}^{\delta}}\frac{\partial^{2} u^{\delta}}{\partial x^{k}\partial x^{h}}\frac{\partial^{2} u^{\nu}}{\partial x^{i}\partial x^{j}}.
 \label{stepbeforedivform}
\end{eqnarray}
This is where we need to use the gradient estimate. We want to show that the right hand side is dominated by the third term whenever $1 - \kappa$ is small enough. First we need to remember that the eigenvalues of $g^{ij}$ are between $1$ and $1/\kappa$, so we have 
\begin{eqnarray}
2g^{ij}(Du)\frac{\partial^{2} u^{\nu} }{\partial x^{i} \partial x^{k}}\frac{\partial^{2}u^{\nu} }{\partial x^{j}\partial x^{k}}  \geq  2|D^{2}u|^{2}. \nonumber  
\end{eqnarray}
Using $g(p) = I - p^{T}p$ and differentiating $g^{-1}g=I$ gives $\partial g^{fh}/\partial p_{k}^{\nu} = g^{kf}g^{jh}p_{j}^{\nu} + g^{kh}g^{fj}p_{j}^{\nu}$, which implies that $\left| \left(\partial g^{fh}/\partial p_{k}^{\nu}\right) \right|(p) \leq 2|g^{-1}|^{2}|p| \leq 2m^{2}|p|/\kappa^{2}$. It is also easy to see that $\left|\partial w/\partial x^{i}\right|   \leq (|\zeta| + 2\sqrt{m})(1-\kappa)^{1/2}|D^{2}u|$, where we have used $|Du|^{2} \leq m|||Du|||^{2} \leq m(1-\kappa)$. We can combine all of the inequalities above and, using the Schwarz inequality, apply them to equation (\ref{stepbeforedivform}) to get  
\begin{eqnarray}
\frac{\partial }{\partial x^{i}}\left(g^{ij}(Du)\frac{\partial w}{\partial x^{j}} \right) \geq 2|D^{2}u|^{2}-\frac{1-\kappa}{\kappa^{2}}C|D^{2}u|^{2},\nonumber
\end{eqnarray}
where the constant $C>0$ depends only on $m$ and $n$ (since $\zeta$ does). So, for $1 - \kappa$ small enough (how small depending on $m$ and $n$), we will have 
\begin{equation}
\frac{\partial }{\partial x^{i}}\left(g^{ij}(Du)\frac{\partial w}{\partial x^{j}} \right)  \geq  0.\label{firstsupersoleqn}
\end{equation}

\begin{lemma}\label{arbdimholdergrad}
Let $\Omega$ be a domain in $\mathbb{R}^{m}$. There exist constants $\kappa,\alpha \in(0,1)$ and $K >0$ such that if a maximal graph in $\mathbb{R}^{m+n}_{n}$ is given by a smooth function $u : \Omega \rightarrow \mathbb{R}^{n}$ with $|||Du|||^{2} \leq 1-\kappa$, then $[Du|_{\Omega'}]_{\alpha} \leq K dist(\Omega', \partial \Omega)^{-\alpha}$ on any subdomain $\Omega'$ with closure contained in $\Omega$. Here $\kappa$, $\alpha$ and $K$ depend on $m$ and $n$.
\end{lemma}
\begin{proof}
We are in a position now where we can directly follow the proof of Theorem 13.6 of \cite{GT}, choosing $\nu$ and $r$ as in the proof of this theorem, but taking $\zeta = \pm 10mn$. The idea is to use the fact that $w$ is a subsolution to a linear equation in divergence form, by (\ref{firstsupersoleqn}), to apply the weak Harnack inequality (Theorem 8.18 of \cite{GT}). This gives estimates on $w$ which imply the required $C^{0,\alpha}$ estimate on $Du$.\end{proof}     

This lemma gives interior estimates, but we need a uniform estimate. Therefore we need a $C^{1,\alpha}$ estimate at the boundary of $\Omega$. To get this, we will need to adjust our problem in such a way that we have a solution (to some elliptic system) which is zero on a flat boundary portion. We will need to be very careful about where $\kappa$ appears in our inequalities. Again, we have to make sure that a strong enough gradient estimate (i.e. $\kappa$ being close to $1$) implies that certain terms dominate. Unfortunately, the fact that we have to transform our domain and boundary data means that the gradient bound needed will depend on $\Omega$ and $\phi$. 

Let $B$ be some ball in $\mathbb{R}^{m}$ with centre on $\partial\Omega$. Taking $B$ to be smaller if necessary, we can assume that there is a coordinate change $F:B\rightarrow F(B)\subset\mathbb{R}^{m}$ such that $F$ and $F^{-1}$ are smooth, with $F(B\cap\partial\Omega)\subset \{y \; | \; y^{m}  = 0\}$ and $F(B\cap\Omega)\subset \{y \; | \; y^{m}  > 0\}$, and such that the matrix $DFDF^{T}$ has eigenvalues between two constants $\Lambda_{F}\geq\lambda_{F}>0$. We define a function $\tilde{u}$ by $\tilde{u}(F(x)) = u(x)$. Then $Du = D\tilde{u}DF$ and, if we define      
\begin{equation}
A_{ij}(y,D\tilde{u}(y)) = \delta_{ij} - \left( \frac{\partial \tilde{u}^{\nu}}{\partial y^{k}}(y) \frac{\partial F^{k}}{\partial x^{i}}(F^{-1}(y))\right)\left(\frac{\partial \tilde{u}^{\nu}}{\partial y^{h}}(y)\frac{\partial F^{h}}{\partial x^{j}}(F^{-1}(y)) \right),\nonumber
\end{equation}
then 
\begin{eqnarray}
 0  =  g^{ij}\frac{\partial^{2}u }{\partial x^{i}\partial x^{j}} =  \left(  \frac{\partial F^{k} }{\partial x^{i}}A^{ij}\frac{\partial F^{h} }{\partial x^{j}} \right)\frac{\partial^{2} \tilde{u}}{\partial y^{k}\partial y^{h}} + A^{ij}\frac{\partial^{2}F^{k}}{\partial x^{i}\partial x^{j}}\frac{\partial \tilde{u} }{\partial y^{k}}.\nonumber
\end{eqnarray}
We define $\tilde{\phi}$ by $\phi = \tilde{\phi}(F)$, and take $\hat{u} = \tilde{u} - \tilde{\phi}$ so that $\hat{u}=0$ on $F(\partial\Omega\cap B)$. 
Taking
\begin{eqnarray}
G^{kh}(y,D\hat{u}) &=& A^{ij}(y, D\hat{u} + D\tilde{\phi})\frac{\partial F^{k}}{\partial x^{i}}\frac{\partial F^{h}}{\partial x^{j}},\nonumber\\
B(y,D\hat{u}) &=& A^{ij}(y, D\hat{u} + D\tilde{\phi})\frac{\partial^{2} F^{k}}{\partial x^{i}\partial x^{j}}\left( \frac{\partial \hat{u}}{\partial y^{k}} + \frac{\partial \tilde{\phi}}{\partial y^{k}} \right) + G^{kh}(y,D\hat{u})\frac{\partial^{2}\tilde{\phi}}{\partial y^{h}\partial y^{k}},\nonumber
\end{eqnarray} 
where the matrix $G^{-1} = (G^{kh})$ has eigenvalues between $\lambda_{F}$ and $\Lambda_{F}/\kappa$, then $\hat{u}$ satisfies the elliptic system 
\begin{equation}
0 = G^{kh}(y,D\hat{u}(y))\frac{\partial^{2} \hat{u}}{\partial y^{k}\partial y^{h}} + B(y,D\hat{u}(y)).\label{transsyst}
\end{equation}
Now we define a function $w = \zeta(1-\kappa)^{1/2}\partial \hat{u}^{\gamma}/\partial y^{r} + \sum_{\nu}\sum_{\ell = 1}^{m-1}\left( \partial \hat{u}^{\nu}/\partial y^{\ell}\right)^{2}$ for some $\gamma \in \{m+1,\ldots,m+n\}$ and $r \in \{1,\ldots,m -1\}$. It is important to remember that we will apply the summation convention over the usual ranges for all indices, except for $\ell = 1,\ldots,m-1$. This means that $w$ only involves the tangential derivatives at the flat boundary portion. We get, by using (\ref{transsyst}) and the same reasoning used to get (\ref{stepbeforedivform}), 
\begin{eqnarray}
\frac{\partial}{\partial y^{i}}\left(G^{ij}\frac{\partial w}{\partial y^{j}} \right) 
& = & \frac{\partial}{\partial y^{i}}(G^{ij})\frac{\partial w}{\partial y^{j}} - \zeta\sqrt{1-\kappa}\frac{\partial}{\partial y^{r}}(B^{\gamma}) \nonumber\\
&& - \zeta\sqrt{1-\kappa}\frac{\partial}{\partial y^{r}}(G^{ij})\frac{\partial^{2}\hat{u}^{\gamma}}{\partial y^{i}\partial y^{j}} + 2G^{ij}\frac{\partial^{2}\hat{u}^{\nu}}{\partial y^{i}\partial y^{\ell}}\frac{\partial^{2}\hat{u}^{\nu}}{\partial y^{j}\partial y^{\ell}} \nonumber\\
&& - 2\frac{\partial \hat{u}^{\nu}}{\partial y^{\ell}}\frac{\partial}{\partial y^{\ell}}(B^{\nu}) - 2\frac{\partial \hat{u}^{\nu}}{\partial y^{\ell}}\frac{\partial}{\partial y^{\ell}}(G^{ij})\frac{\partial^{2}\hat{u}^{\nu}}{\partial y^{i}\partial y^{j}}.\label{uglytransformeddivform}
\end{eqnarray}    
Note that we will also stop labelling constants here and will, for now, just denote by $C$ any positive constant depending only on $m$, $n$, $F$ and $\Phi_{2}$ (but not $\kappa$). By the same reasoning as before, we easily see that 
\begin{eqnarray}
&&\left|\left(\frac{\partial A^{ij}}{\partial y^{k}}\right) \right|  \leq  |A^{-1}|^{2}\left|\left(\frac{\partial A_{ij}}{\partial y^{k}} \right)\right|  \leq \frac{C}{\kappa^{2}}\left|\left(\frac{\partial F^{h}}{\partial x^{i}} p_{h}^{\nu}p_{f}^{\nu}\frac{\partial }{\partial y^{k}}\frac{\partial F^{f}}{\partial x^{j}} \right)\right|\leq \frac{C}{\kappa^{2}}|p|^{2},\nonumber\\
&&\left|\left(\frac{\partial A^{ij}}{\partial p^{\nu}_{k}}\right) \right| \leq |A^{-1}|^{2}\left|\left(\frac{\partial A_{ij}}{\partial p^{\nu}_{k}} \right)\right| \leq \frac{C}{\kappa^{2}}\left|\left(\frac{\partial F^{h}}{\partial x^{i}}\delta_{\nu\eta}\delta_{h k}p^{\nu}_{f}\frac{\partial F^{f}}{\partial x^{j}}\right)\right| \leq \frac{C}{\kappa^{2}}|p|.\nonumber
\end{eqnarray}
We can use these inequalities and $|D\tilde{u}| \leq |Du|\cdot|DF| \leq C\sqrt{1-\kappa} \leq C$ (along with the Schwarz, Young and triangle inequalities, and $0<\kappa<1$) to get
\begin{eqnarray}
\left|\frac{\partial }{\partial y^{k}}\left( G^{ij}(y, D\hat{u}(y)) \right)\right| +
\left|\frac{\partial }{\partial y^{k}} \left(B^{\nu}(y, D\hat{u}(y))\right)  \right|  \leq   \frac{C}{\kappa^{2}}(|D^{2}\hat{u}|+1) +  \frac{C}{\kappa^{2}}|D^{3}\tilde{\phi}|.\nonumber
\end{eqnarray}
We also have $|Dw| \leq C |D^{2}\hat{u}|\sqrt{1-\kappa}$. We apply all of these inequalities to equation (\ref{uglytransformeddivform}) to get  
\begin{eqnarray}
\frac{\partial}{\partial y^{i}}\left(G^{ij}\frac{\partial w}{\partial y^{j}} \right)& \geq & 2G^{ij}\frac{\partial^{2}\hat{u}^{\nu}}{\partial y^{i}\partial y^{\ell}}\frac{\partial^{2}\hat{u}^{\nu}}{\partial y^{j}\partial y^{\ell}} - \frac{C\sqrt{1-\kappa}}{\kappa^{2}}\left(|D^{2}\hat{u}|^{2} + 1 + |D^{3}\tilde{\phi}|\right)\nonumber\\
& \geq &  2\lambda_{F}\sum_{\nu,i,\ell}\left(\frac{\partial^{2}\hat{u}^{\nu}}{\partial y^{i}\partial y^{\ell}}\right)^{2}   - \frac{C\sqrt{1-\kappa}}{\kappa^{2}}\left(|D^{2}\hat{u}|^{2} + 1 +\Phi_{3}\right),\label{dominantermineq}  
\end{eqnarray}    
and we hope that the first term will dominate when $\kappa$ is close enough to $1$. Obviously this term contains all second order derivatives of $\hat{u}$ except $\partial^{2}\hat{u}/\partial x^{m}\partial x^{m}$. By using system (\ref{transsyst}), and the obvious bounds on $|B|$ and $|G^{-1}|$, we can estimate this remaining second order derivative, 
\begin{eqnarray}
\left|\frac{\partial^{2}\hat{u}^{\nu}}{\partial y^{m}\partial y^{m}}\right|^{2} = \left|\sum_{(i,j)\neq(m,m)}\frac{G^{ij}}{G^{mm}}\frac{\partial^{2}\hat{u}^{\nu}}{\partial y^{i}\partial y^{j}} - \frac{B^{\nu}}{G^{mm}}\right|^{2}  \leq  \frac{C}{\kappa^{2}}\left(\sum_{i,\ell}\left(\frac{\partial^{2}\hat{u}^{\nu}}{\partial y^{i}\partial y^{\ell}}\right)^{2}  + 1\right),\nonumber
\end{eqnarray}
where we have used Young's inequality and $\kappa<1$. This implies that
\begin{eqnarray}
|D^{2}\hat{u}|^{2}  \leq  \frac{C}{\kappa^{2}}\left(\sum_{\nu,i,\ell}\left(\frac{\partial^{2}\hat{u}^{\nu}}{\partial y^{i}\partial y^{\ell}}\right)^{2} + 1 \right).\label{boundonfulbytang}
\end{eqnarray}
The right hand side again contains all derivatives except $\partial^{2}\hat{u}/\partial x^{m}\partial x^{m}$. Combining this with inequality (\ref{dominantermineq}) gives
\begin{eqnarray}
\frac{\partial}{\partial y^{i}}\left(G^{ij}\frac{\partial w}{\partial y^{j}} \right) \geq C\left(\kappa^{2}|D^{2}\hat{u}|^{2} -1\right) - \frac{C\sqrt{1-\kappa}}{\kappa^{2}}\left(|D^{2}\hat{u}|^{2} + 1 + \Phi_{3}\right),\nonumber
\end{eqnarray}    
which implies that if we choose $\kappa$ close enough to $1$ then the terms involving $|D^{2}\hat{u}|^{2}$ will cancel, and then $w$ will be a subsolution to some linear elliptic equation in divergence form. It is important to note that our choice of $\kappa$ is determined only by $m$, $n$, $\Omega$ (through dependence on $F$) and $\Phi_{2}$. 

\begin{lemma}\label{unifholderonderiv}
Let $\Omega$ be a smooth, bounded domain in $\mathbb{R}^{m}$, and let $\phi:\bar{\Omega} \rightarrow \mathbb{R}^{n}$ be a smooth function with $||\phi||_{2}\leq \Phi_{2}$ and $||\phi||_{3}\leq \Phi_{3}$ for some constants $\Phi_{2},\Phi_{3}>0$. There exist constants $\kappa,\alpha\in(0,1)$ and $K>0$ such that if a maximal graph in $\mathbb{R}^{m+n}_{n}$ is given by a smooth function $u : \bar{\Omega} \rightarrow \mathbb{R}^{n}$, with $u|_{\partial\Omega}=\phi|_{\partial\Omega}$ and $|||Du|||^{2} \leq 1-\kappa$, then $[Du]_{\alpha} \leq K$. Here $\kappa$ depends on $m$, $n$, $\Phi_{2}$ and $\Omega$, while $\alpha$ and $K$ depend on $m$, $n$, $\Omega$ and $\Phi_{3}$. 
\end{lemma}

\begin{proof}
Since $w$ is a subsolution, we have enough to directly follow the proof of Theorem 13.7 in \cite{GT}. Since $F(B\cap\partial\Omega)$ is flat, and since $\hat{u}$ is zero on this boundary portion, we know that the tangential derivatives $\partial \hat{u}/\partial y^{\ell}$ will be zero there, and therefore so will $w$. Using this fact, we apply the boundary weak Harnack inequality (Theorem 8.26 of \cite{GT}) exactly as in section 13.4 of \cite{GT}, getting H\"older estimates on each $\partial \hat{u}/\partial y^{\ell}$. The estimate on $\partial \hat{u}/\partial y^{m}$ comes from the estimates on $\partial \hat{u}/\partial y^{\ell}$ by using inequality (\ref{boundonfulbytang}). This gives us H\"older estimates on $u$ near the boundary of $\Omega$, which combine with the interior estimates from Lemma \ref{arbdimholdergrad} to complete the proof.    
\end{proof}      

\begin{theorem}\label{arbdimcodimexistthm}
Given a convex, smooth, bounded domain $\Omega$ in $\mathbb{R}^{m}$, there will exist a constant $C$ (depending on $\Omega, m, n$) such that the maximal graph Dirichlet problem in $\mathbb{R}^{m+n}_{n}$ will have a smooth solution, with $u|_{\partial\Omega} = \phi|_{\partial\Omega}$, for any smooth $\phi:\bar{\Omega}\rightarrow\mathbb{R}^{n}$ with $C^{2}$ norm less than $C$.\footnote{Also, if we know various properties of the domain (curvature of $\partial\Omega$, etc.) then it would be possible (but difficult and tedious) find a lower bound on the required $C$, by following the proofs of the $C^{1,\alpha}$ estimates more carefully to check the constants involved at each step.}   
\end{theorem}
\begin{proof}
Let $||\phi||_{2} \leq \Phi_{2}$ and $||\phi||_{3}\leq\Phi_{3}$. Let $\kappa=\kappa(m,n,\Omega,\Phi_{2})$ be as in Lemma \ref{unifholderonderiv}. Assume further that $||\phi||_{2}$ is small enough that $|||D\phi||^{2}\leq 1-\kappa$ and that inequality (\ref{boundaryassumpt}) holds for this $\kappa$. This gives the gradient estimate needed to apply Lemma \ref{unifholderonderiv}, giving a $C^{1,\alpha}$ estimate. These clearly hold for solutions with boundary values $\sigma\phi$, for any $\sigma \in [0,1]$, allowing us to apply Lemma \ref{firstlemma}.  
\end{proof}

\section{The Minimal Graph Dirichlet Problem}
\label{TheMinimalGraphProblem}

Our main theorems say that a solution to the maximal graph Dirichlet problem will exist for boundary data with $C^{2}$ norm less than some constant. It seems obvious to attempt to improve these results in such a way that the constant will depend on the domain in a simpler way. Such a result (but for minimal graphs) is the goal of \cite{MTW2}, where the proof uses a gradient estimate and White's regularity theorem (see \cite{BW}) for mean curvature flows. For spacelike mean curvature flows in semi-Euclidean spaces, we have a similar gradient estimate (see Proposition \ref{mcfgrad}) and a version of White's theorem (see \cite{BST}), 
 so it makes sense to attempt a similar proof. In \cite{MTW2}, White's theorem is used to get $C^{2,\alpha}$ estimates to prove long time existence. However, White's theorem only gives \emph{local} estimates, which are not enough to prove long time existence in this problem, leaving a gap in the proof. It is not clear if the proof in \cite{MTW2} can be fixed, but we can get a similar (but weaker) result by repeating the proofs of the $C^{1,\alpha}$ estimates used in Theorem \ref{arbdimcodimexistthm} (now in the minimal graph case) and using the gradient estimate from \cite{MTW2}.   

\begin{theorem}
Let $\Omega$ be a bounded, smooth and convex domain in $\mathbb{R}^{m}$. There is a positive constant $C$ (depending on $\Omega, m, n$) such that there exists a smooth solution to the minimal graph Dirichlet problem in $\mathbb{R}^{m+n}$, with boundary values $\phi|_{\partial\Omega}$, whenever $\phi: \bar{\Omega} \rightarrow \mathbb{R}^{n}$ is smooth with $C^{2}$ norm less than $C$.
\end{theorem}      

Acknowledgements: The author is grateful to Professor Mu-Tao Wang for his help and, in particular, for the suggestion of how to obtain a priori H\"older estimates on the gradient for the case of dimension greater than $2$. The results proved here formed part of the author's PhD thesis at Durham University, under the supervision of Dr Wilhelm Klingenberg.  

\appendix
\section{Gradient Estimate for Mean Curvature Flow}
\label{sec:GradientEstimateForSpacelikeMeanCurvatureFlow}

Our goal in this section will be to prove a gradient estimate for spacelike mean curvature flows satisfying certain boundary/initial conditions. Suppose that we have a graphic mean curvature flow in $\mathbb{R}^{m+n}_{n}$, given by some function $u:\Omega\times (0,T)\rightarrow \mathbb{R}^{n}$ for a bounded, convex, $C^{2}$ domain $\Omega \subset \mathbb{R}^{m}$. We assume $u$ is smooth on the interior of its domain and $C^{1}$ on the closure. We take the induced metric $(g_{ij})$ from $\mathbb{R}^{m+n}_{n}$ on spatial slices $M_{t} = \{(x,u(x,t))\in \mathbb{R}^{m+n}_{n} \; | \; x\in\Omega\}$ for each $t\in[0,T]$, and we assume that these are spacelike (i.e. that $|||Du||| < 1$, where $D$ taken with respect to the space variables in $\mathbb{R}^{m}$ only). By the mean curvature flow condition, $u$ satisfies the parabolic system $\partial u / \partial t = g^{ij}(Du)\partial^{2}u/\partial x^{i} \partial x^{j}$.

\begin{proposition}\label{mcfgrad}Let $\phi : \bar{\Omega}\times[0,T] \rightarrow \mathbb{R}^{n}$ be a  $C^{2}$ function and let $\kappa\in(0,1)$. If the function $u$ above satisfies the boundary/initial condition that $u(x,t) = \phi(x,t)$ whenever $x\in\partial\Omega$ or $t = 0$, then the inequality $\sup_{\Omega}|||Du|||^{2} < 1 - \kappa$ will hold for all times in $[0,T]$ if the (parabolic) $C^{2}$ norm of $\phi$ is small enough. 

\end{proposition}
\begin{proof}
For $\phi$ small enough in $C^{2}$, we can assume that $\sup_{\Omega}|||Du(\cdot,0)|||^{2}=\sup_{\Omega}|||D\phi(\cdot,0)|||^{2} < 1 - \kappa^{1/m}<1-\kappa$. Suppose that there exists some first time $\epsilon \in (0,T]$ such that $|||Du(\cdot,\epsilon)|||^{2} = 1 - \kappa$ for some point in $\bar{\Omega}$. Then we have $|||Du|||^{2} \leq 1 - \kappa$ on $\bar{\Omega}\times[0,\epsilon]$. Now we take the linear parabolic operator $L = \partial/\partial t - g^{ij}(Du)\partial^{2}/\partial x^{i}\partial x^{j}$, and we define $S^{\pm}$ exactly as we did earlier (where $d$ is still a function of the space variables on $\bar{\Omega}$ only, independent of the time variable). We easily see that
\begin{eqnarray}
LS^{\pm}=\frac{\nu\zeta^{2}}{(1 + \zeta d)^{2}}g^{ij}\frac{\partial d}{\partial x^{i}}\frac{\partial d}{\partial x^{j}} \pm L\phi^{\gamma}  \geq   \frac{\nu\zeta^{2}}{(1 + \zeta \textrm{diam}\Omega)^{2}} -  \left| \frac{\partial \phi}{\partial t}\right| - \frac{m}{\kappa}|||D^{2}\phi||| \geq 0\nonumber
\end{eqnarray}
if we choose $\zeta = 1/\textrm{diam}\Omega$ and $\nu\zeta = 4\textrm{diam}\Omega\sup_{\Omega\times(0,T)}(|\partial\phi/\partial t| + m|||D^{2}\phi|||/\kappa)$. Then we can apply the parabolic maximum principle to the inequalities to again (exactly as in the maximal graph case) get $|||Du||| \leq \nu\zeta + 2|||D\phi|||$ on $\partial\Omega$. Then, for $\phi$ with small enough parabolic $C^{2}$ norm, we will have $\sup|||Du|||^{2} < 1 - \kappa^{1/m}$ on the parabolic boundary $\partial\Omega\times [0,\epsilon]\cup \Omega\times \{0\}$. Since we already know (by the definition of $\epsilon$) that $|||Du|||^{2} \leq 1 - \kappa$ on $\bar{\Omega}\times[0,\epsilon]$, we can use $\left(d/dt - \Delta_{M_{t}}\right)\log\sqrt{\det g} \geq 0$, which follows from the proof of Proposition 5.2 in \cite{LS2}. By applying the parabolic maximum principle to this inequality, we see that $|||Du|||^{2} < 1 - \kappa$ on $\bar{\Omega}\times [0,\epsilon]$. This contradicts the definition of $\epsilon$. Therefore, if the $C^{2}$ norm of $\phi$ is as small as described, the gradient estimate $|||Du|||^{2}<1 -\kappa$ will hold for all times for which the flow exists. 
\end{proof}
Paying closer attention to this proof, we see that $\sup_{\Omega}|||D\phi(\cdot,0)||| < \sqrt{1 - \kappa^{1/m}}$ and $4\textrm{diam}\Omega\sup_{\Omega\times(0,T)}(|\partial\phi/\partial t| + m|||D^{2}\phi|||/\kappa)+2\sup_{\partial\Omega}|||D\phi|||<\sqrt{1-\kappa^{1/m}}$ will be enough for the gradient estimate to hold.


\begin{thebibliography}{}

\bibitem{BS} R~Bartnik and L~Simon, {\sl Spacelike Hypersurfaces with Prescribed Boundary Values and Mean Curvature}, Commun. Math Phys. 87, 131-152 (1982)
\bibitem{FJF} F~J~Flaherty, {\sl The Boundary Value Problem for Maximal Hypersurfaces}, Proc. Natl. Acad. Sci. USA, Vol. 76, No. 10, 4765-4767 (1979)
\bibitem{GT} D~Gilbarg and N~Trudinger, {\sl Elliptic Partial Differential Equations of Second Order}, Springer, Berlin (1998) 
\bibitem{JS} H~Jenkins and J~Serrin, {\sl The Dirichlet Problem for the Minimal Surface Equation in Higher Dimension}, J. Reine Angew. Math. 229, 170-187 (1968)
\bibitem{LO} H~Lawson and R~Osserman, {\sl Non-Existence, Non-Uniqueness and Irregularity of Solutions to the Minimal Surface System}, Acta Math. 139 (1977)
\bibitem{LS} G~Li and I~Salavessa, {\sl Graphic Bernstein Results in Curved Pseudo-Riemannian Manifolds}, J. Geom. Phys. 59, 1306-1313 (2009)
\bibitem{LS2} G~Li and I~Salavessa, {\sl Mean Curvature Flow of Spacelike Graphs}, to appear in Math. Z (accepted in May 2010)
\bibitem{BO} B~O'Neill, {\sl Semi-Riemannian Geometry}, Academic Press, New York (1983)
\bibitem{BST} B~S~Thorpe, {\sl A Regularity Theorem for Graphic Spacelike Mean Curvature Flows}, to appear in Pacific Journal of Mathematics (accepted in October 2011) 
\bibitem{MTW2} M~T~Wang, {\sl The Dirichlet Problem for the Minimal Surface System in Arbitrary Codimension}, Comm. Pure. Appl. Math. 57, no.2 (2004), 267-281 
\bibitem{BW} B~White, {\sl A Local Regularity Theorem for Mean Curvature Flow}, Ann. Math. 161, 1487-1519 (2005)  

\end{thebibliography}
\end{document}